\documentclass[11pt,letterpaper]{amsart}

\usepackage[centertags]{amsmath}
\usepackage{amsthm}
\usepackage{amssymb}
\usepackage{amscd}
\usepackage{eucal}
\usepackage{epsfig}
\usepackage{verbatim}
\usepackage{color}

\usepackage{latexsym,enumerate,amsxtra}

\newtheorem{thm}{Theorem}[subsection]
\newtheorem{theorem}[thm]{Theorem}
\newtheorem{corollary}[thm]{Corollary}
\newtheorem{lemma}[thm]{Lemma}
\newtheorem{proposition}[thm]{Proposition}
\newtheorem{defn}[thm]{Definition}

\newtheorem{remarks}[thm]{Remarks}
\newtheorem{remark}[thm]{Remark}
\newtheorem{example}[thm]{Example}

\newtheorem{numbering}[thm]{}

\newcommand{\nin}{\newline\indent{\ }}

\newcommand{\xyinj}[1][r]{\ar@{^(->}[#1]}

\def\XYmatrix{\xymatrix@M=5pt}

\newcommand{\CaC}{\mathcal C}

\newcommand{\CaO}{\mathcal O}
\newcommand{\CaH}{\mathcal H}

\newcommand{\CaL}{\mathcal L}
\newcommand{\CaR}{\mathcal R}

\newcommand{\CaW}{\mathcal W}

\newcommand{\CaV}{\mathcal V}

\newcommand{\frs}{{\mathfrak s}}
\newcommand{\wfrs}{{\widetilde{\mathfrak s}}}

\newcommand{\frB}{{\mathfrak B}}

\newcommand{\bbC}{\mathbb C}
\newcommand{\bbF}{\mathbb F}

\newcommand{\bbN}{\mathbb N}
\newcommand{\bbO}{\mathbb O}

\newcommand{\bbR}{\mathbb R}

\newcommand{\rT}{\mathrm T}

\newcommand{\sfK}{\mathsf K}

\newcommand{\orb}{\mathcal O}

\newcommand{\orbp}{\mathcal O^\prime}

\newcommand{\Tr}{\textup{Tr}}

\newcommand{\bG}{\mathsf G}
\newcommand{\bH}{\mathsf H}

\newcommand{\bGL}{\mathsf{GL}}

\newcommand{\G}{G}

\newcommand{\GL}{GL}
\newcommand{\rM}{M}

\newcommand{\wM}{\widehat\rM}
\newcommand{\wT}{\widehat\rT}

\newcommand{\wG}{\G^\vee}
\newcommand{\wGp}{\G^{\prime\vee}}
\newcommand{\walpha}{\widehat\alpha}
\newcommand{\Ldual}{{{}^L\!\Omega}}
\newcommand{\LdualG}{{{}^L\!\G}}
\newcommand{\LdualGp}{{{}^L\!\G'}}

\newcommand{\lieG}{\mathfrak g}
\newcommand{\lieH}{\mathfrak h}

\newcommand{\blieG}{\boldsymbol{\lieG}}
\newcommand{\blieH}{\boldsymbol{\lieH}}

\newcommand{\sms}{{s\!s}}

\newcommand{\gss}{\gamma}
\newcommand{\sgss}{\Gamma}
\newcommand{\refs}{s}

\newcommand{\dth}{{\mathrm d}}

\newcommand{\charpi}{\Theta_\pi}

\newcommand{\Ccs}{C_c^\infty}

\newcommand{\nil}{\mathcal N}

\newcommand{\Hom}{\textup{Hom}}
\newcommand{\Gal}{\textup{Gal}}

\newcommand{\irr}{\textsf{Irr}}
\newcommand{\irrG}{\textsf{Irr}(G)}
\newcommand{\irrGp}{\textsf{Irr}(G')}
\newcommand{\irrGn}{\textsf{Irr}(G_n)}
\newcommand{\irrGm}{\textsf{Irr}(G_m)}

\newcommand{\vol}{\textit{vol}}

\newcommand{\wtB}{\widetilde\frB}

\usepackage[top=2.7cm, bottom=2.7cm, left=2.7cm, right=2.7cm]{geometry}

\newcommand \fk[1]{{{\mathfrak #1}}}

\newcommand \wti[1]{{\widetilde {#1}}}

\newcommand\fg{\mathfrak g}

\newcommand\CG{{\mathcal G}}

\newcommand\CI{{\mathcal I}}
\newcommand\CO{{\mathcal O}}

\newcommand\CP{{\mathcal P}}

\newcommand\al{{\alpha}}

\newcommand\AZ{{\mathsf{AZ}}}
\newcommand\bfGL{{\mathsf{GL}}}

\newcommand\bfG{{\mathsf G}}

\newcommand\End{\operatorname{End}}

\newcommand\Ind{\operatorname{Ind}}

\newcommand\St{\mathsf{St}}

\newcommand\Irr{\mathsf{Irr}}
\newcommand\Rep{\mathsf{Rep}}

\newcommand\supp{\mathsf{supp}}

\newcommand\ind{\mathsf{ind}}

\newcommand\reg{\mathsf{reg}}

\newcommand\WF{\mathsf{WF}}

\newcommand\Ad{\operatorname{Ad}}

\def\<{\langle} 
\def\>{\rangle}

\numberwithin{equation}{section}

\begin{document}

\title[The ($\sgss$-asymptotic) Wavefront Sets: $\GL_n$]{The ($\sgss$-asymptotic) Wavefront Sets: $\GL_n$}

\author
{Dan Ciubotaru}
        \address[D. Ciubotaru]{Mathematical Institute, University of Oxford, Oxford OX2 6GG, UK}
        \email{dan.ciubotaru@maths.ox.ac.uk}

\author
{Ju-Lee Kim}
        \address[J.-L. Kim]{Department of Mathematics, M.I.T., Cambridge MA 02139, USA}
        \email{juleekim@mit.edu}


\begin{abstract} 
{Let $\G$ be a connected reductive $p$-adic group. As verified for unipotent representations, it is expected that there is a close relation between the (Harish-Chandra-Howe) wavefronts sets of irreducible smooth representations and their Langlands parameters in the local Langlands correspondence via the Lusztig-Spaltenstein duality and the Aubert-Zelevinsky duality. In this paper, we define the $\sgss$-asymptotic wavefront sets generalizing the notion of wavefront sets via the $\sgss$-asymptotic expansions (\cite{KM}), and then study the their relation with the Langlands parameters. When $G=GL_n$, it turns out that this reduces to the corresponding relation of unipotent representations of the appropriate twisted Levi subgroups via Hecke algebra isomorphisms. For unipotent representations of $GL_n$,  we also  describe the Harish-Chandra--Howe (HCH) local character expansions of irreducible smooth representations using Kazhdan-Lusztig theory, and give another computation of the coefficients in the HCH expansion and the wavefront sets.  
}
\end{abstract}

\date{\today}

\maketitle

\section{\bf Introduction}

Let $F$ be a nonarchimedean field with residue field $\mathbb F_q$. 
Let $W_F$ be the Weil group with inertia subgroup $I_F$ and wild inertia subgroup $P_F$. Let $W'_F$ be the Weil-Deligne group, i.e., $W'_F=W_F\ltimes \bbC$, where the action is defined by $w\cdot z=||w|| z$, $w\in W_F$, $z\in \bbC$ where $||~||$ is the norm on $W_F$.

Let $\bfG$ be a reductive $F$-group that splits over an unramified extension with its Lie algebra $\blieG$. 
For every $\omega\in H^1(F,\bfG)$, let $\G^\omega=\bfG^\omega(F)$ denote the corresponding pure inner form (in the sense of \cite{vogan-llc}). Let $G^\vee$ be the complex dual group to $\bfG$. Let ${}^L\!G=G^\vee\rtimes W_F$ be the $L$-group of $\G$.

\subsection{}{\bf Bernstein center.}\label{sec:1.1} Let $\CaR(\G^\omega)$ be the category of smooth representations of $\G^\omega$. Then $\CaR(\G^\omega)$ is the direct product of full abelian subcategories $\CaR^\frs(\G^\omega)$ parameterized by an inertia class $\frs$ of cuspidal data $(M,\sigma)$ which consists of a Levi subgroup $M$ and its supercuspidal representation $\sigma$ (see \cite{BD,BK} for details):
$\CaR(\G^\omega)=\prod_{\frs\in\frB(\G^\omega)}\CaR^\frs(\G^\omega)$ where $\frB(\G^\omega)$ is the set of inertia classes of $G^\omega$. 

Likewise, the set $\irr(\G^\omega)$ of irreducible smooth representations of $\G^\omega$ is a disjoint union of $\irr^\frs(\G^\omega)=\CaR^\frs(\G^\omega)\cap\irr(\G^\omega)$, $\frs\in\frB(\G^\omega)$: $\irr(\G^\omega)=\bigcup_{\frs\in\frB(\G^\omega)}\irr^\frs(\G^\omega)$.
Recall from \cite{KM} that when $p\gg0$, one can associate a semisimple element $\gss_\frs\in\lieG^\omega$ (up to rational conjugacy modulo compact elements) such that each $\pi\in\irr^\frs(\G^\omega)$ has a $\gss_\frs$-asymptotic expansion in the neighborhood of $1$ (see \cite{KM} and Definition \ref{def: endscopic wf}). Set $\blieG_\sms=\bigcup_{\omega}\lieG^\omega_\sms$ where $\lieG^\omega_\sms$ denotes the set of semisimple elements in $\lieG^\omega$. For $\gss\in\blieG_\sms/\!\!\sim$ (geometric conjugacy classes modulo compact set), let $\frB^\gss(\G^\omega)=\{\frs\in\frB(\G^\omega)\mid \gss_\frs\sim\gss\}$ where $\gss_\frs\sim\gss$ if $\gss_\frs$ and $\gss$ are rationally conjugate modulo compact elements in $\lieG^\omega$. Note that $\frB^\gss(\G^\omega)$ is finite and possibly empty. Let $\irr^\gss(\bfG)=\bigcup_\omega\bigcup_{\frs\in\frB^\gss(\G^\omega)}\irr^\frs(\G^\omega)$.
Then, we have
\begin{equation}\label{dec-gp}
\irr(\bfG)=\bigcup_{\gss\in\blieG_\sms/\sim}\irr^\gss(\bfG)
\end{equation}

Looking at the Galois side, let $\Ldual_\bfG$ be the set of Langlands parameters $\varphi:W'_F\rightarrow\LdualG$. We write $\varphi_\sms$ for $\varphi|W_F$ and $\mathbb O^\vee_\varphi$ for $G^\vee$-saturation of $\varphi(1_\bbC)$. Let ${}^L\!\frB_\bfG$ the set of stable inertia classes (see \cite{Ha} or \cite{AMS} for definition). Then, each $\wfrs\in{}^L\!\frB_\bfG$ parameterizes a subset $\Ldual^\wfrs_\bfG$ and 
$\Ldual_\bfG\,=\,\overset\circ\bigcup_{\wfrs\in{}^L\!\frB_\bfG}\,\Ldual^\wfrs_\bfG$, a disjoint union.

For a given semisimple parameter $\varphi_{\sms}:W_F\rightarrow\LdualG$ (\cite{FS, AMS}), one can associate a finite subset ${}^L\!\frB_\bfG^{\varphi_{\sms}}$ of ${}^L\!\frB_\bfG$ such that for any $\phi\in\bigcup_{\wfrs\in{}^L\!\frB_\bfG^{\varphi_{\sms}}}\Ldual^\wfrs_\bfG$, $\phi|W_F$ is $\G^\vee$-conjugate to $\varphi_{\sms}|I_F$.
Similarly, for $\varphi^+=\varphi_\sms|P_F=\varphi|P_F$, we can define $\frB^{\varphi^+}_\bfG$.
Then we have
\begin{equation}\label{dec-gal}
\Ldual_\bfG\,=\,\overset\circ\bigcup_{\varphi^+}\,\Ldual^{\varphi^+}_\bfG
\end{equation}

Since $\gamma$ (resp. $\varphi^+$) encodes the ramified part of a representation (resp. Langlands parameter), we conjecture that the local Langlands correspondence $\CaL:\irr(\bfG)\rightarrow\Ldual_\bfG$ is compatible with the decompositions in \eqref{dec-gp} and \eqref{dec-gal}. In a sense, we expect that for given $\gss\in\blieG_\sms$, there is an equivalence class of $\varphi^+$ such that $\CaL$ maps $\irr^\gss(\bfG)$ surjectively onto $\Ldual^{\varphi^+}_\bfG$, giving ``partitions" on $\irr(\bfG)$ and $\Ldual_\bfG$ compatible with the local Langlands correspondence.

Now considering the nilpotent part of Langlands parameters, recall from \cite{CMBO} that wavefront sets of representations and the nilpotent orbits $\bbO_\varphi\subset\nil^\vee$ are closely related via $\CaL$: when $\pi\in\irrG$ is a unipotent representation (in this case $\varphi_\sms=1$ and $\gss=0$), the geometric wavefront set of $\pi$ can be bounded by the Spaltenstein dual orbit of $\bbO^\vee_{\varphi_{\AZ(\pi)}}$.
Motivated by \cite{CMBO}, we will look at an interplay between $\gss$-asymptotic wavefront sets of irreducible representations (see Definition \ref{def: endscopic wf}) and the nilpotent part of Langlands parameters.

\subsection{}{\bf Wavefront sets and Kazhdan-Lusztig theory for $\GL_n$.}

To an irreducible smooth representation $\pi$, the Harish-Chandra-Howe (HCH) local character expansion describes the character $\Theta_\pi$  as a linear combination of Fourier transform of nilpotent orbital integrals $\widehat{\mu_\orb}$: there are a small $G$-domain $\CaV_0$ around 0 and $c_\orb(\pi)\in\bbC$, $\orb\in\orb(0)$ such that 
 \[
\charpi(\exp X)=\sum_{\orb\in \orb(0)}\, 
c_\orb(\pi)\, \widehat{\mu_\orb}(X),\qquad X\in\CaV_0\cap\lieG^\reg
\]
where $\orb(0)$ is the set of nilpotent orbits in $\lieG$. When $p\gg0$, one can take $\CaV_0=\lieG_{\dth(\pi)^+}$ where $\lieG_{\dth(\pi)^+}$ is a $\G$-domain of elements in $\lieG$ of depth larger than the depth $\dth(\pi)$ of $\pi$. Then, the wavefront set $\WF(\pi)$ of $\pi$ is the set of maximal elements in $\{\orb\in\orb(0)\mid c_\orb\ne0\}$, which is a fundamental invariant of $\pi$. 

For Iwahori spherical representations with real infinitesimal character, their geometric wavefront sets are computed in \cite{CMBO}. Those are the representations in $\irr^{\frs_0}(\G)$ where 
$\frs_0=(T,1)\in\frB(G)$ where $T$ is a maximal torus of $G$ and the type associated to $\frs_0$ is the trivial representation of an Iwahori subgroup $I$ of $\G$. To phrase the main result from {\it loc. cit.}, we need to invoke two known constructions. The first is Spaltenstein duality \cite{Spa}, a map \[d^\vee:\G^\vee\backslash\nil^\vee\rightarrow \bG\backslash\nil\] between the sets of nilpotent orbits in $\mathfrak g^\vee$ and in $\mathfrak g$ whose image is the set of special nilpotent orbits and $(d^\vee)^3=d^\vee$. The second is the involution $\AZ$ on the set of irreducible smooth representations, first defined by Zelevinsky \cite{Zel} for $GL_n(F)$ and generalized by Aubert \cite{Au-inv} to arbitrary groups $\G$. If $\pi$ is an irreducible Iwahori spherical representation, or more generally, a unipotent representation in the sense of Lusztig, with real infinitesimal character, then it is shown in \cite{CMBO} that
 \[{}^{\bar F}\WF(\AZ(\pi))=d^\vee(\bbO_\pi^\vee),\]
  where $\mathbb O_\pi^\vee$ is the $G^\vee$-saturation of the nilpotent Langlands parameter of $\pi$.  The proof of this formula relies on the Barbasch-Moy \cite{BM-local} test functions and Okada's \cite{Ok} refinement, together with an analysis of the maximal parahoric restrictions of representations, via Hecke algebra isomorphisms and difficult case by case calculations of the restrictions of Springer representations.
  
In general, the upper bound conjecture (\cite[\S1]{CK}, \cite{HLLS}) proposes that 
\[{}^{\bar F}\WF(\AZ(\pi))\le d^\vee(\bbO_\pi^\vee),\]
always holds (the inequality is with respect to the closure ordering).

For $\GL_n$, the upper bound conjecture is valid, in fact a stronger statement holds. 
\begin{theorem}\cite[II.2]{MW}\label{thm: MW-intro} For every irreducible $\mathsf{GL}_n(F)$-representation $\pi$, with the same notation as above,
\begin{equation}\label{e:gl-intro}
    {}^{\bar F}\WF(\AZ(\pi))=d^\vee(\bbO_\pi^\vee).
    \end{equation}
\end{theorem}

In \S2, we describe HCH-expansion using Kazhdan-Lusztig theory and compute both sides of \ref{e:gl-intro}, thus give another proof of Theorem \ref{thm: MW-intro}.

\subsection{}{\bf $\sgss$-asymptotic wavefront sets and Langlands parameters}

While any smooth representations have HCH expansions, $\sgss$-asymptotic expansions are more adapted to each Bernstein component. 
Let $\gss\in\lieG^\omega_\sms$ and $\orb(\gss)$ the set of $\G$-orbits in $\lieG$ whose closures contain $\gss$. For $\pi\in\irr(\G^\omega)$, if there exist complex numbers $c_{\orb}(\pi)$, $\orb\in \orb(\gss)$ and a small neighborhood $\CaV$ of $0$ such that
\begin{equation}\label{eq: gss-expansaion}
\charpi(\exp X)=\sum_{\orb\in \orb(\gss)}\, 
c_{\orb}(\pi)\, \widehat{\mu_{\orb}}(X),\qquad X\in\lieG_{(\frac{\dth(\pi)}2)^+}\cap\lieG^\reg
\end{equation}
we say that $\Theta_\pi$ is $\sgss$-{\it asymptotic with respect to }$\gss$ or $\gss$-{\it asymptotic} on $\CaV$.\footnote{While {\it $\Gamma$-asymptotic expansion} is a general term to refer to an expression in \ref{eq: gss-expansaion}, {\it $\gss$-asymptotic expansion} is a specialization to $\Gamma=\gss$.} Similarly as in the HCH case, one can define the $\gss$-asymptotic wave front set $\WF(\pi,\gss)$ (see Definition \ref{def: endscopic wf}). 

When $p\gg0$, for $\pi\in\irr^\frs(\G^\omega)$, it is shown in \cite{KM} that there is $\gss_\frs$ such that $\Theta_\pi$ is $\gss_\frs$-asymptotic on $\lieG_{(\frac{\dth(\pi)}2)^+}$.
In fact, $\gss_\frs$ depends only on $\frs$. We will also write $\gss_\pi$ for $\gss_\frs$ when appropriate.

For example, for a depth zero representation, $\gss_\pi$ is equivalent to $0$ and thus the $\gss_\pi$-asymptotic expansion of a depth zero representation reduces to the HCH expansion. For positive depth representations, $\gss_\pi$ is nontrivial. Note that $\gss_\pi$-expansions are valid on a larger $\G$-domain in general than the HCH expansion. Moreover, since $\gss_\pi$ encodes the structure of $\varphi_\pi|_{P_F}$ (and vice versa \cite{Kal, Kal2}) and $\bbO_{\varphi_\pi}$ is the orbit of $\varphi_\pi(1_\bbC)\in C_{\lieG^\vee}(\varphi_\pi^+)$, 
one can expect that $\WF(\pi,\gss_\pi)$ is more closely related to $\varphi_\pi$. 

When $\G=\GL_n$, $\gss_\pi$ can be refined to another semisimple element $\refs_\pi$ with $C_\G(\gss_\pi)\supseteq C_\G(\refs_\pi)$ so that $\pi$ is $\refs_\pi$-asympototic on $\lieG_{(\frac{\dth(\pi)}2)^+}$ and  Theorem \ref{thm: Howe-Moy} holds for $\G'=C_\G(\refs_\pi)$ (Remark \ref{rmk: refined ss}). Then, there corresponds an irreducible representation $\pi'$ of $\G'$ via the Hecke algebra isomorphism in Theorem \ref{thm: Howe-Moy}. For example, when $\pi$ is supercuspidal, the refined $\refs_\pi$ is elliptic regular semisimple, the associated Hecke algebra is commutative, and the corresponding $\pi'$ is a character of an anisotropic torus. 
The relation between $\refs_\pi$-asymptotic wavefront sets and nilpotent parts of Langlands parameters can be more explicitly formulated when $\bG=\mathsf{GL}_n$ (see Theorem \ref{thm: Main2}).

In the following, if $\CaL(\pi)=\varphi_\pi$ (resp. $\CaL(\pi')=\varphi_{\pi'}$), we write $\bbO_\pi^\vee$ (resp. $\bbO_{\pi'}^{'\vee}$) for 
the $\G^\vee$-saturation (resp. $\G^{'\vee}$-saturation) of $\varphi_\pi(1_\bbC)$ (resp. $\varphi_{\pi'}(1_\bbC)$).

\begin{thm}\label{thm: Main2} Suppose $\pi\in\irr(GL_n)$ contains a pure type (see Definition \ref{def: pure type}). Let $\refs_\pi\in\lieG^\sms$ be associated to the pure type and $G'=C_G(\refs_\pi)$. Let $\pi'\in\irrGp$ be the corresponding unipotent representation of $G'$. Then,
\[
\WF(\pi,\refs_\pi)= d'(\mathbb O^{'\vee}_{\AZ(\pi')}).
\]
\end{thm}

Here, $\AZ(\pi')$ is the Aubert-Zelevinsky dual of $\pi'$, $\mathbb O^{'\vee}_{\AZ(\pi')}$ is the $G^{'\vee}$-saturation of the nilpotent part of the Langlands parameter of $\AZ(\pi')$ in $\lieG^{'\vee}$ and $d':\G^{'\vee}\backslash\nil^{'\vee}\rightarrow\bG'\backslash\nil'$ is the Spaltenstein dual map (see \S4).

\

 In \S2, we use Kazhdan--Lusztig theory for $GL_n$ to compute wavefront sets of irreducible representations and thus to give another proof of Theorem \ref{thm: MW-intro} in the case of unipotent representations. In the process, we also find explicit formulae for all the coefficients in the HCH expansion in terms of Kazhdan-Lusztig multiplicities.
 
In \S3 and \S4, we review necessary backgrounds to formulate and prove Theorem \ref{thm: Main2}. In \S3, we review the Bernstein--Zelevinsky classification and its relation to Langlands parameters. In \S4, we discuss a rationality question regarding the Lusztig-Spaltenstein duality. This step is necessary to analyze the duality in twisted Levi subgroups of $\G$. Lastly, \S5 focuses on proving Theorem \ref{thm: Main2}.

\

\noindent
{\sc  More Notation and Conventions.} We keep the notation from the above. We use bold face letters $\bG$, $\bH$ etc denote reductive groups over $F$, and corresponding roman letters the group of $F$-rational points. In particular, let $\bG_n$ denote $\bGL_n$. Similarly, $\blieG$, $\blieH$, $\lieG$, $\lieH$ etc denote the Lie algebras. By $X_{\sms}$, we denote the subset of semisimple elements in $X\subset\lieG$. 

We let $\CG$ to denote a reductive group defined over an algebraically closed field. We will still use $\lieG$ to denote its Lie algebra if there is no confusion.

\

\ 

\section{\bf Wavefront Sets: $\bfGL_n$}\label{sec: KL} 

\subsection{}\label{subsec: Spaltenstein} {\bf Spaltenstein duality.}
Let $\mathfrak g$ and $\mathfrak g^\vee$ be two reductive Lie algebras over an algebraically closed field of characteristic $0$ with root systems dual to each other. 
Let $\CG$, $\CG^\vee$ be two connected reductive groups with Lie algebras $\mathfrak g$ and $\mathfrak g^\vee$, respectively. 
Let $\mathcal N$ and $\mathcal N^\vee$ be the nilpotent cones, the varieties of $ad$-nilpotent elements in $\mathfrak g$ and  $\mathfrak g^\vee$, respectively. 
The group $\CG$ acts on $\mathcal N$ by the adjoint action and there finitely many orbits whose classification is independent of the isogeny of $\CG$ (same for $\CG^\vee$).  Spaltenstein duality \cite{Spa} is an assignment

\begin{equation}
    d: \CG^\vee\backslash \mathcal N^\vee\longrightarrow \CG\backslash \mathcal N
\end{equation}
which is not injective or surjective in general, but whose image consists of the {\it special orbits} in the sense of Lusztig. By abuse of notation, denote by $d$ also the Spaltenstein duality in the opposite direction. When restricted to the set of special orbits on both sides, $d$ is a duality, in the sense that it is bijection and $d^2=\text{id}$. On the full set of orbits, it is still true that $d^3=d$.

\begin{example}\rm 
If $\mathfrak g=\mathfrak{gl}_n$, the nilpotent orbits in $\lieG$ and $\lieG^\vee$ are parameterized by partitions $\lambda$ of $n$ via the Jordan normal form and $d(\lambda)=\lambda^t$, where $\lambda^t$ is the transpose partition.
\end{example}

\subsection{} {\bf Upper bound conjecture}
For $\GL_n=\bfGL_n(F)$, the upper bound conjecture \cite[\S1.5]{CK} is true. In fact, a stronger statement holds. 
\begin{theorem}[{M\oe glin-Waldspurger \cite[II.2]{MW}}] \label{thm: MW} For every irreducible $\GL_n$-representation $\pi$,
\begin{equation}\label{e:gl}
    {}^{\bar F}\WF(\AZ(\pi))=d^\vee(\mathbb O^\vee_\pi)
    \end{equation}
\end{theorem}

The calculation in \cite{MW} shows that the wavefront set of a representation of $\GL_n$ is independent of the supercuspidal support (which is always generic for $\GL_n$). In other words, for $\GL_n$, the calculation of the wavefront set for an irreducible representation in an arbitrary Bernstein component is not essentially different than that in the Iwahori component. 

For this reason, we now give a different proof of (\ref{e:gl}) in the Iwahori-spherical case which helps to explain the appearance of the duality $\AZ$ in the formulae for the wavefront set. The ``philosophy" is that, at least for $\GL_n$, in some precise sense, the local character expansion is the $\AZ$-dual of the character identity expressing an irreducible representations in terms of standard Langlands representations.

\subsection{}
In the rest of \S\ref{sec: KL}, let $\G=\GL_n$. Via the Jordan normal form, the set of nilpotent orbits in $\fg$ is in a one-to-one correspondence $\CO_\lambda\leftrightarrow \lambda$ with the set of partitions $\mathcal P(n)$ of $n$. By a partition $\lambda\in \mathcal P(n)$ we understand a collection of positive integers $\lambda=(\lambda_1,\dots,\lambda_\ell)$ such that $\lambda_1\ge \lambda_2 \ge\dots\ge \lambda_\ell$ and $\sum_i \lambda_i=n$. If $\lambda\in \mathcal P(n)$, let us denote by $e_\lambda=e_{\CO_\lambda}$ the representative given by the upper-triangular nilpotent Jordan blocks with sizes given by the parts $\lambda_j$ of $\lambda=(\lambda_1,\dots,\lambda_\ell)$. The closure ordering on nilpotent orbits is the same as the dominance order on partitions:
\[\CO_\lambda\subset \overline \CO_\mu \text{ iff } \lambda\le \mu\text{ iff } \sum_{1\le i\le t} \lambda_i\le \sum_{1\le i\le t} \mu_i,\text{ for all }t.
\]

Write $(\GL_m)^k_\Delta$ for the group $\GL_m$ embedded diagonally in the direct product of $k$ copies of $\GL_m$. For every $i\ge 1$, denote $r_i=|\{j\mid \lambda_j=i\}|$, so that, in particular, $\sum_{i\ge 1} i r_i=n$. The centralizer of $e_\lambda$ in $G$ is
\begin{equation}
Z_\lambda= Z_\lambda^{\mathsf{red}} U_\lambda,\text{ where }Z_\lambda^{\mathsf{red}}=\prod_{i\ge 1} (\GL_{r_i})^i_\Delta\cong \prod_{i\ge 1} \GL_{r_i},
\end{equation} 
and $U_\lambda$ is a unipotent group.

In the case of $G$, the distributions $\widehat \mu_{\CO_\lambda}$ in the local character expansion, $\lambda\in \mathcal P(n)$, admit a very concrete interpretation \cite{Ho}. Define a block upper-triangular parabolic subgroup $P(\lambda)$ as follows. Let $V=F^n$ be the standard representation of $G$ and for every $j\le 1$, set $V_j=\ker(e_\lambda^j)$. This defines a partial flag $(0=V_0\subseteq V_1\subseteq V_2\subseteq\dots)$ of $V$, and let $P(\lambda)$ be the stabilizer of this flag, $P(\lambda)=\{x\in G\mid x\cdot V_j\subseteq V_j,\ j\ge 1\}$. The corresponding Lie algebra decomposes $\fk p(\lambda)=\fk m(\lambda)\oplus \fk n(\lambda)$, where $\fk m(\lambda)$ is a block-diagonal Levi subalgebra and $\fk n(\lambda)$ is the nilpotent radical. By definition, $e_\lambda\in \fk n(\lambda)$ and the $P(\lambda)$-orbit of $e_\lambda$ in $\fk n(\lambda)$ is open dense. Moreover, 
\begin{equation}
\fk m(\lambda)\cong \bigoplus_{j\ge 1} \mathfrak{gl}_{\lambda_j'}, \text{ where }\lambda^t=(\lambda_1',\lambda_2',\dots)
\end{equation}
is the transpose partition to $\lambda$. Note that
$P(\lambda)=P_{\lambda^t}$ and $M(\lambda)=M_{\lambda^t}$

Combining \cite[Proposition 5]{Ho} and homogeneity results in \cite{Wa, De},
\begin{equation}
\widehat\mu_{\CO_\lambda}(X)=\Theta_{\Ind_{P(\lambda)}^G(\mathbf 1)}(1+X),\text{ for all } X\in \fg_{0+},
\end{equation}
where $\Ind_P^G$ denotes the functor of (unnormalized) parabolic induction and $\mathbf 1$ is the trivial representation. Consequently, the $0$-asymptotic character expansion can be written in this case as the identity of distributions
\begin{equation}\label{e:lc-GL}
\Theta_\pi=\sum_{\lambda\in \CP(n)} c_{\CO_\lambda}(\pi) ~\Theta_{\Ind_{P(\lambda)}^G(\mathbf 1)}, \text{ valid on the space } C_c^\infty(G_{\dth(\pi)+}),
\end{equation}
where $\dth(\pi)$ is the depth of $\pi$ and $\G_{r+}$ for $r\ge0$ is the set of all elements in $\G$ of depth greater than $r$.

\subsection{} Let us restrict now to the case when $\dth(\pi)=0$ and consider representations $\pi$ with fixed vectors under an Iwahori subgroup. Let $K_0=\bfGL_n(\CaO_F)$ be the maximal hyperspecial subgroup  and let $I\subset K_0$ be the Iwahori subgroup given by the pull-back of the upper triangular Borel subgroup under the projection map $K_0\to \overline K_0=\bfGL_n(\bbF_q)$. Let $(\pi,V_\pi)$ be an irreducible smooth $G$-representation such that $V_\pi^I\neq 0$.

We recall the Langlands classification of irreducible $\G$-representations $\Irr_I G$ with $I$-fixed vectors. Let $\St_n$ denote the Steinberg representation of $\G$. 
The representation $\St_n$ is square-integrable (modulo the center). Let $\al$ be a composition of $n$, i.e., $\al$ is a tuple of positive integers (not necessarily in non-increasing order) $(\al_1,\dots,\al_\ell)$ such that $\sum \al_i=n$. Let $M_\al=\prod_{i=1}^\ell \GL_{\al_i}$ be the corresponding block-diagonal Levi subgroup and let $P_\al\supset M_\al$ be the upper-triangular parabolic subgroup. For every collection of complex numbers $\nu=(\nu_1,\dots,\nu_\ell)$ such that $\Re \nu_1\ge \Re \nu_2\ge\dots\ge \Re\nu_\ell$, define the standard representation
\begin{equation}\label{standard-gl}
\CI(\al;\nu)=\ind_{P_\al}^G(\St_{\al_1} |\det|_{F}^{\nu_1}\boxtimes\dots\boxtimes \St_{\al_\ell} |\det|_{F}^{\nu_\ell})
\end{equation}
where $\ind_P^G$ denotes the functor of normalized parabolic induction. 
This standard representation has a unique irreducible quotient $\pi(\al;\nu)$ and every irreducible representation in $\Irr_IG$ is obtained uniquely in this way. 

The representation $\pi(\al;\nu)$ is $K_0$-spherical if and only if $\al=(1,1,\dots,1)$, in which case $\CI(\al;\nu)$ is an unramified principal series (with dominant parameter $\nu$).

If $\nu_1=\dots=\nu_\ell=0$, then write $\CI(\al;0)$ for the standard representation. This is irreducible and tempered. In particular, $\pi(\al;0)=\CI(\al;0)$.

Moreover, in the Grothendieck group, or equivalently, as the level of distribution characters, one may write
\begin{equation}\label{mult-gl}
\Theta_{\pi(\al;\nu)}=\Theta_{\CI(\al,\nu)}+\sum_{\al',\nu'} m((\al;\nu),(\al',\nu'))~\Theta_{\CI(\al',\nu')},
\end{equation}
for some integers $m((\al;\nu),(\al',\nu'))$. To make this more precise, we need to recall the geometric classification of \cite{KL}.

\medskip

Let $G^\vee=\bfGL_n(\bbC)$ be the Langlands dual complex group with Lie algebra $\fg^\vee$ and maximal diagonal torus $T^\vee$ with Cartan subalgebra $\fk t^\vee$. The geometric Kazhdan-Lusztig classification \cite{KL} says that there is a natural one-to-one correspondence the set of $G^\vee\text{-orbits in }\{(\tau^\vee,e^\vee)\mid \tau^\vee\in T^\vee,\ e^\vee\in \fg^\vee,\ \Ad(\tau^\vee) e^\vee = q^{-1} e^\vee\}$ and $\Irr_I G$:
\[
(\tau^\vee,e^\vee)\leftrightarrow \pi(\tau^\vee,e^\vee). 
\]
Let $\fk t^\vee_\bbR$ be the real part of $\fk t^\vee$. For every $s^\vee\in \fk t^\vee_\bbR$ and $e^\vee\in\fg^\vee$ such that $[s^\vee,e^\vee]=-e^\vee$, we also write
\[\pi(q^{s^\vee},e^\vee)\text{ for }\pi(s^\vee,e^\vee),\text{ where } q^{s^\vee}=\exp(\log q\cdot s^\vee).
\]

\begin{example}\rm In this notation, we have
\[
\begin{array}{llll}
\pi(-\rho^\vee,e_{(1,\dots,1)})&=&\mathbf 1_G &\text{(the trivial representation)},\\ 
\pi(-\rho^\vee,e_{(n)})&=&\St_G &\text{(the Steinberg representation)}.
\end{array}
\]
Here $\rho^\vee\in\fk t^\vee_\bbR$ is the half-sum of positive coroots of $G^\vee$. 
\end{example}

Let $G^\vee(s^\vee)$ denote the centralizer of $s^\vee$, which is Levi subgroup. For a fixed $s^\vee\in \fk t^\vee_\bbR$, if we restrict to the representations of type $\pi(s^\vee,e^\vee)\in\Irr_I(\G)$ with $e^\vee\in\lieG^\vee_{-1}$, then  there is a one-to-one correspondence between irreducible representations $\pi(s^\vee,e^\vee)$ and 
\[G^\vee(s^\vee)\text{-orbits in }\fg^\vee_{-1}=\{X^\vee\in \fg^\vee\mid [s^\vee,X^\vee]=-X^\vee\},\quad \pi(s^\vee,e^\vee)\leftrightarrow e^\vee.
\]
There are finitely many $G^\vee(s^\vee)$-orbits in $\fg^\vee_{-1}$ and hence a unique open orbit.

To describe the dictionary with the parabolic Langlands classification as above (where we now assume $\nu_1,\dots,\nu_\ell\in\bbR$), let $\overline\al$ denote the partition of $n$ obtained by permuting the entries of $\al=(\al_1,\al_2,\cdots,\al_\ell)$. Let $(\frac{a-1}2,\frac{a-3}2,\dots,-\frac{a-1}2)$ denote the coordinate entries of $\rho^\vee$ in $\bfGL_a(\bbC)$. Let 
\[s^\vee_\nu=(-\frac{\alpha_1-1}2+\nu_1,-\frac{\alpha_1-3}2+\nu_1,\dots,\frac{\alpha_1-1}2+\nu_1,\dots,-\frac{\alpha_\ell-1}2+\nu_\ell,-\frac{\alpha_\ell-3}2+\nu_\ell,\dots,\frac{\alpha_\ell-1}2+\nu_\ell),
\]
viewed as an element of $\fk t^\vee_\bbR.$ Let $e^{\vee}_{\al}$ denote the nilpotent element given by the upper-triangular Jordan blocks of sizes corresponding to the entries of $\al$. Then indeed 
\[[s^\vee_\nu,e^{\vee}_{\al}]=-e^{\vee}_\al.\]
Notice that $e^{\vee}_\al$ belongs to the nilpotent orbit $\CO_{\overline\al}$. The correspondence is:
\begin{equation}
\pi(\al;\nu) = \pi(s^\vee_\nu,e^{\vee}_{\al}).
\end{equation}
In relation to the identity (\ref{mult-gl}), the coefficients $m((\al;\nu),(\al',\nu'))$ admit a geometric interpretation \cite{zelevinsky-p-adic}, via a version of Kazhdan-Lusztig polynomials, which in particular implies that
\begin{equation}\label{mult-geom}
m((\al;\nu),(\al',\nu'))\neq 0 \text{ only if } s^\vee_{\nu'} \text{ is $S_n$-conjugate to } s^\vee_{\nu}\text{ and } e_\al\in \overline \CO_{\overline{\al'}}.
\end{equation}

\smallskip

In $\G=\GL_n$, the set of compact elements $\G_0$ are the $\G$-orbit of $K_0$: $G_{0}=G\cdot K_0$ (this is due to the fact that all maximal compact subgroups of $\GL_n$ are conjugate). Since the unramified characters are trivial on compact elements, equation (\ref{mult-gl}) implies that for all $\al$ and $\nu_1,\dots,\nu_\ell\in \bbR$:
\[\Theta_{\pi(\al;\nu)}=\Theta_{\CI(\al,0)}+\sum_{\al',\nu'} m((\al;\nu),(\al',\nu'))~\Theta_{\CI(\al',0)} \text{ on } C^\infty_c(G_{0}).
\]
Moreover, the tempered characters $\Theta_{\CI(\al',0)}$ only depend on the partition $\overline{\al'}$ and not on the composition $\alpha'$ itself. Taking also into account (\ref{mult-geom}), we see that there exist integers 
\[
\wti m(\al,\lambda;\nu)=\sum_{\overline{\al'}=\lambda,\nu'} m((\al;\nu),(\al',\nu'))
\]
 such that
\begin{equation}
\Theta_{\pi(\al;\nu)}=\Theta_{\CI(\overline \al,0)}+\sum_{\lambda\in \mathcal P(n),\overline \al<\lambda} \wti m(\al,\lambda;\nu) ~\Theta_{\CI(\lambda,0)} \text{ on } C^\infty_c(G_{0^+}).
\end{equation}
The involution $\AZ: \Irr_I(G)\to \Irr_I(G)$ commutes with parabolic induction and it swaps $\St_G$ and $\mathbf 1_G$. This means in particular that
\begin{equation}
\Theta_{\AZ(\CI(\lambda,0))}=\Theta_{\ind_{P_\lambda}^G(\mathbf 1)}=\Theta_{\Ind_{P(\lambda^t)}^G(\mathbf 1)}\text{ on } C^\infty_c(G_{0^+}).
\end{equation}
Note that the characters of normalized and unnormalized induction are same on $G_{0^+}$. This gives the identity
\begin{equation*}
\Theta_{\AZ(\pi(\al;\nu))}=\Theta_{\Ind_{P(\overline\al^t)}^G(\mathbf 1)}+\sum_{\lambda\in \mathcal P(n),\overline \al<\lambda} \wti m(\al,\lambda;\nu) ~\Theta_{\Ind_{P(\lambda^t)}^G(\mathbf 1)}.
\end{equation*}
Since the transpose on partitions reverses the dominance order, we may rewrite this as
\begin{equation}\label{IM-char-gl} 
\Theta_{\AZ(\pi(\al;\nu))}=\Theta_{\Ind_{P (\overline \al^t)}^G(\mathbf 1)}+\sum_{\mu\in \mathcal P(n),\mu<\overline\al^t} \wti m(\al,\mu^t;\nu) ~\Theta_{\Ind_{P(\mu)}^G(\mathbf 1)}\quad \text{ on } C^\infty_c(G_{0^+}).
\end{equation}
For consistency, we may set $\wti m(\al,\overline\al)=1$. 
\begin{theorem}
Let $\al=(\alpha_1,\cdots,\alpha_\ell)\in\mathcal P(n)$, $\nu=(\nu_1,\dots,\nu_\ell)$, $\nu_i\in \bbR$ with $\nu_1\ge\nu_2\ge\dots\ge\nu_\ell$. Let $\pi(\al;\nu)$ be the irreducible Langlands representation which is the quotient of the standard representation $\CI(\al;\nu)$ in (\ref{standard-gl}). Then the local character expansion of the $\AZ$-dual of $\pi(\al;\nu)$ is given by (\ref{IM-char-gl}), hence
\[ c_{\CO_\lambda}(\AZ(\pi(\al;\nu))=\wti m(\al,\lambda^t;\nu).
\]
 In particular, the wavefront set orbit of $\AZ(\pi(\al;\nu))$ is $\CO_{\overline\al^t}$ and the leading coefficient is $1$.
\end{theorem}

\begin{proof}
We have $G_{0+}\subset G_0$ (as an open and closed subset), hence (\ref{IM-char-gl}) holds on $C_c^\infty(G_{0+})$ by restriction. Comparing with (\ref{e:lc-GL}) the claim follows since the distributions $\Theta_{\Ind_{P(\lambda)}^G(1)}$ are linearly independent on $G_{0+,\sms}^\reg$.
\end{proof}

Now, Theorem \ref{thm: MW} also follows from the above, since $\mathbb O^\vee_{\pi(\alpha;\nu)}=\mathbb O^\vee_{\overline\alpha}$, and $d^\vee(\mathbb O^\vee_{\overline\alpha})=\CO_{\overline\al^t}$.

\begin{example}\rm
Suppose $\pi(\al;\nu)$ is generic in the sense of admitting (nondegenerate) Whittaker models. It is known that this is the case if and only if $\pi(\al;\nu)=\CI(\al;\nu)$. Assume this is the case. Since 
\[\pi(\al;\nu)=\CI(\al;\nu)=\AZ(\ind_{P_\al}^G(|\det|^{\nu_1}\boxtimes\dots\boxtimes |\det|^{\nu_\ell}))=\AZ(\pi((1,\dots,1);\nu)),\]
it follows that the wave front set of $\pi(\al;\nu)$ is indeed $(1,\dots,1)^t=(n)$, a well-known result of Rodier.
\end{example}

\begin{example}\rm
Let $\CO^\vee_\lambda$ be a nilpotent orbit in $\fg^\vee$. Fix a Lie triple $\{e^\vee_\lambda,h^\vee_\lambda,f^\vee_\lambda\}$, $e^\vee_\lambda\in \CO^\vee_\lambda$ and $h^\vee_\lambda\in \fk t^\vee_\bbR$. Denote
\[\pi_{\mathsf{sph}}(\CO^\vee)=\AZ(\pi(-\frac 12 h^\vee_\lambda,e^\vee_\lambda)),
\]
a $K_0$-spherical $G$-representation. We call this a \emph{special (spherical) unipotent representation}. It is easy to see that
\[\pi_{\mathsf{sph}}(\CO^\vee_\lambda)\cong \ind_{P_\lambda}^G(\mathbf 1),
\]
which implies that its local character expansion is
\begin{equation}
\Theta_{\pi_{\mathsf{sph}}(\CO^\vee_\lambda)}(1+X)=\widehat\mu_{\CO_{\lambda^t}}(X), \text{ for all regular semisimple }X\in \fg_{0+}.
\end{equation}
In particular, the wave front set orbit of $\pi_{\mathsf{sph}}(\CO^\vee_\lambda)$ is $\CO_{\lambda^t}$.
\end{example}

\

\section
{\bf Bernstein--Zelevinsky classification and Langlands parameters}

\subsection{} {\bf Classification of Irreducible representations of $\G_n=GL_n$}\ 

In this section, we review Bernstein--Zelevinsky classification of smooth irreducible representations of $\G_n$. We fix a torus $T$ and a Borel subgroup $B$ containing $T$.
For $m\in\bbN$, let $\nu_m$ denote the character of $\G_m$ given by $|\det(\cdot)|$. When there is no confusion, we will simply write $\nu$ for $\nu_m$. Let $\CaC_m$ denote the set of all supercuspidal representations of $\G_m$, and $\CaC=\cup_m\CaC_m$. For $\rho\in\CaC_m$, $k\in\bbN$, let $\Delta=\Delta(\rho,k)$ denote the segment $\langle\rho,\nu\rho,\nu^2\rho,\cdots,\nu^{k-1}\rho\rangle$. Let 
$\rho\times\nu\rho\times\cdots\times\nu^{k-1}\rho$ denote 
$\ind_{P_{m,k}}^{\G_{mk}}\rho\otimes\cdots\otimes\nu^{k-1}\rho$ where $P_{m,k}$ is the block upper-triangular parabolic subgroup with Levi subgroup isomorphic to $\prod_{i=1}^k\G_m$.
Let $\langle\Delta\rangle$ (resp. $\langle\Delta\rangle^t$) denote the unique irreducible submodule (resp. quotient) of  $\rho\times\nu\rho\times\cdots\times\nu^{k-1}\rho$. For $\Delta(\rho,k)$ and $\Delta(\rho',k')$, we say that $\Delta(\rho,k)$ precedes $\Delta(\rho',k')$ if $\rho'=\nu^\alpha\rho$ for some $\alpha\in\bbN$.

\begin{thm} (\cite[Thm 6.1]{Zel})
\begin{enumerate}
\item 
Let $\Delta_1,\cdots,\Delta_r$ be segments in $\mathcal C$. Suppose for each pair of indices $i,j$ such that $i<j$, $\Delta_i$ does not precede $\Delta_j$. Then the representation $\langle\Delta_1\rangle\times\cdots\times \langle\Delta_r\rangle$ has a unique irreducible submodule; denote it by $\langle\Delta_1,\cdots,\Delta_r\rangle$.
\item
The representations $\langle\Delta_1,\cdots,\Delta_r\rangle$ and 
$\langle\Delta'_1,\cdots,\Delta'_s\rangle$ are isomorphic if and only if the sequences 
$(\Delta_1,\cdots,\Delta_r)$ and $(\Delta'_1,\cdots,\Delta'_s)$ are equal up to a rearrangement.
\item
Any irreducible representation of $\G_n$ is isomorphic to some representation of the form $\langle\Delta_1,\cdots,\Delta_r\rangle$.
\end{enumerate}
\end{thm}

Let $\Sigma$ be the set of all segments and $\Theta$ be the set of all finite multisets on $\Sigma$.
Any $\theta=\{\Delta_1,\Delta_2,\cdots, \Delta_r\}\in\Theta$ can be rearranged so that the condition in (1) is satisfied. Let $\tau(\theta)$ be the associated irreducible representation. The support of a segment $\Delta(\rho,k)$ is the set $\supp(\Delta)=\{\rho,\nu\rho,\nu^2\rho,\cdots,\nu^{k-1}\rho\}$. The support of a multisegment $\theta$ is the disjoint union of the supports of its segments, $\supp(\theta)=\overset\circ\bigcup_i~\supp(\Delta_i)$.

\begin{remarks}\rm \ 
\begin{enumerate}
\item 
{If $\pi=\langle\Delta_1,\cdots,\Delta_r\rangle$ and $\Delta_i=\Delta(\rho_i,k_i)$ with $\rho_i\in\CaC_{m_i}$, then $\pi\in\irr^\frs(\G_n)$ where $\frs$ is the inertia class of $\left(M_{(m_1^{k_1},\cdots,m_r^{k_r})}, \,\bigotimes_{i=1}^r(\otimes^{k_i}\rho_i)\right)$.}

\item
For $\theta\in\Theta$, if $\supp(\theta)$ is contained in $\{\nu^\alpha\rho\mid \alpha\in\bbC\}$ for a fixed supercuspidal representation $\rho$ of $\G_m$, $\tau(\theta)$ contains a {\it pure} unrefined minimal $\sfK$-type (see Definition \ref{def: pure type}).
\end{enumerate}
\end{remarks}

\subsection{} {\bf Representations of Weil groups} \cite[\S10]{Zel}\ 

Let $W_F$ denote the Weil group of $F$. Let $\nu'$ be the character of $W_F$ corresponding to the norm character of $F^\times$.
Denote by $\CaW$ the set of isomorphism classes of pairs $(\sigma, N)$ where $\sigma$ is an algebraic completely reducible representation of $W_F$ on the $\bbC$-vector space $V$ and $N:V\rightarrow V$ belongs to $\Hom_{W_F}(\nu'\sigma,\sigma)$.
Denote by $\CaW_n$ be the subset of $(\sigma,N)\in\CaW$ with $\dim(\sigma)=n$.
The local Langlands correspondence states that there is a natural bijection between $\irrGn$ and $\CaW_n$. Since there is one to one correspondence between $\Ldual$ and $\CaW_n$ (where $\Ldual$ is as in \S\ref{sec:1.1}), the local Langlands correspondence map $\irrGn\ni\omega\mapsto(\sigma(\omega),N(\omega))\in\CaW_n$ has following properties:

\smallskip

\begin{enumerate}
\item 
$\nu\omega$ corresponds to the pair $(\nu'\sigma(\omega),N(\omega))$, that is, 
$(\sigma(\nu\omega),N(\nu\omega))=(\nu'\sigma(\omega),N(\omega))$.
\item
$\omega$ is cuspidal if and only if $\sigma(\omega)$ is irreducible (and thus $N(\omega)=0$).
\item
If ${\rm supp}(\omega)=\{\rho_1,\cdots,\rho_r\}$, then 
$\sigma(\omega)
\simeq\sigma(\rho_1)\oplus\cdots\oplus\sigma(\rho_r)$.
\end{enumerate}

\smallskip

Denote by $\CaC'$ be the set of equivalence classes of irreducible finite dimensional representations of $W_F$. A segment in $\CaC'$ is a subset $\Delta'=[\sigma,\sigma']\subset\CaC'$ of the form $\Delta'=\{\sigma,\nu'\sigma,\nu^{\prime2}\sigma,\cdots,\nu^{\prime k-1}\sigma=\sigma'\}$ for $k\in\bbN$. Let $\Sigma'$ be the set of all segments in $\CaC'$ and $\Theta'$ be the set of all finite multisets on $\Sigma'$.

To each segment $\Delta'=[\sigma,\sigma']\in\Sigma'$, we assign
\[
\tau(\Delta')=(\sigma(\Delta'),N(\Delta'))\in\CaW
\]
where $\sigma(\Delta')=\sigma\oplus\nu'\sigma\oplus\cdots\oplus\nu^{\prime k-1}\sigma$ and $N(\Delta')\in\Hom(\nu'\sigma(\Delta'),\sigma(\Delta'))$ of maximal rank. In this case, if $\dim(\sigma)=m$, $\tau(\Delta')\in\CaW_{mk}$ and $N(\Delta')$ is the nilpotent orbit of $\mathfrak{gl}_{mk}$ corresponding to the partition $m^k=(m,m,\cdots,m)$ of $mk$.

\begin{lemma} \ 
\begin{enumerate}
\item The objects $\tau(\Delta')$ $(\Delta'\in\Sigma)$ are indecomposable and mutually non-isomorphic, and each indecomposable object of $\CaW$ is of this form.
\item Each object of $\CaW$ decomposes into the direct sum $\tau(\Delta_1')\oplus\cdots\oplus\tau(\Delta_r')$. This decomposition is unique up to permutation.
\end{enumerate}
\end{lemma}

Denote the set of all finite multisets on $\Sigma'$ by $\Theta'$. To $\theta'\in\Theta'$, let
\[
\tau(\theta')=\sum_{\Delta'\in \theta'}\tau(\Delta')\in\CaW.
\]
By the above proposition, $\theta'\mapsto\tau(\theta')$ is a bijection between $\Theta'$ and $\CaW$.

\subsection{\bf Local Langlands Correspondence}\label{subsec: LLC}\ 

In \cite[\S10]{Zel}, the Langlands correspondence is described modulo supercuspidal representations. More precisely,
if $\theta\in\Theta$ and $\theta'\in\Theta'$ correspond to each other under the bijection $\Theta\rightarrow\Theta'$ via local Langlands correspondence of supercuspidal representations, then $\tau(\theta')\in\CaW$ corresponds to $\AZ(\langle \theta\rangle)\in\cup_m\irrGm$.

\ 

\section{\bf Twisted Levi subgroups and Spaltenstein duality}

\subsection{\bf Spaltenstein duality and outer automorphisms}

\ 

Let $\CG$, $\CG^\vee$, $\mathfrak g$ and $\mathfrak g^\vee$ be as in \S\ref{subsec: Spaltenstein}. 
Since the nilpotent elements live in the derived subalgebra of $\mathfrak g$, without loss of generality, we can assume from now on that $\mathfrak g$ (also $\mathfrak g^\vee$) is semisimple.

The group of outer automorphisms of $\mathfrak g$ can be identified with the group of graph automorphisms of the Dynkin diagram of $\mathfrak g$, hence $\text{Out}(\mathfrak g)=\text{Out}(\mathfrak g^\vee)$. If $\tau\in \text{Out}(\mathfrak g)$ is an automorphism of $\mathfrak g$, denote by $\tau^\vee$ the corresponding automorphism of $\mathfrak g^\vee$.

\begin{proposition}\label{prop: nil duality}
    For every $\mathcal O^\vee\in \mathcal G^\vee\backslash \mathcal N^\vee$ and every $\tau\in \mathrm{Out}(\mathfrak g)$,
    \[
    d(\tau^\vee(\mathcal O^\vee))=\tau (d(\mathcal O^\vee)).
    \]
\end{proposition}

\begin{proof}
Assume first that $\mathfrak g$ is a simple Lie algebra. The only cases when $\text{Out}(\mathfrak g)$ is nontrivial are when the Dynkin type is $A_{n-1}$, $D_n$, or $E_6$, in which cases $\text{Out}(\mathfrak g)=\mathbb Z/2.$

If $\mathfrak g=\mathfrak{sl}(n)$, the nontrivial outer automorphism can be realized as $\tau(A)=-A^t.$ If $A$ is a Jordan normal form with partition $\lambda$, it is well known that $-A$ is conjugate to $A$, and the transpose is also conjugate. This means that $\tau$ fixes every nilpotent orbit, so there is nothing to prove.

If $\mathfrak g=E_6$, the nontrivial outer automorphism $\tau$ flips the Dynkin diagram about the branching node. The classification of nilpotent orbits via weighted Dynkin diagrams, see \cite[\S 8.4]{CMc}, shows that all weighted diagrams of type $E_6$ are symmetric with respect to this flip. It follows that again, $\tau$ fixes all of the nilpotent orbits.

If $\mathfrak g=D_n$, the orthogonal Jordan normal form classification says that the nilpotent $O(n)$-orbits are in one-to-one correspondence with partitions $\lambda$ of $2n$, where each even part appears with even multiplicity. Every such orbit $\mathcal O_\lambda$ forms a single $SO(n)$-orbit unless $\lambda$ is a {\it very even partition} \cite[\S 5.3]{CMc}, in which case it splits into two $SO(n)$-orbits labelled $\mathcal O_\lambda^I$
 and $\mathcal O_\lambda^{II}$.
 For $D_n$, the nontrivial outer automorphism $\tau$ acts on the Dynkin diagram by flipping the two extremal nodes connected to the branching node. Using the weighted Dynkin diagram classification \cite[\S 5.3]{CMc}, we see that the only time when $\tau$ does not fix a nilpotent orbit is if $\lambda$ is a very even partition, and in this case $\tau(\mathcal O_\lambda^I)=\mathcal O_\lambda^{II}.$ The description of $d$ on orthogonal partitions \cite[Corollary 6.3.5]{CMc} implies that $d$ maps not very even partitions to not very even partitions, and if $\lambda$ is a very even partition then
 \begin{equation*}
d(\mathcal O_\lambda^{I})=\begin{cases}\mathcal O_\lambda^{I}, &n \text{ even}\\
\mathcal O_\lambda^{II}, &n \text{ odd}\end{cases},
 \end{equation*}
 and similarly for $d(\mathcal O_\lambda^{II})$. This means that $d$ commutes with $\tau$ in this case too.

 Now suppose $\mathfrak g$ is semisimple and it contains two factors isomorphic to a simple algebra $\mathfrak h$ on which $\tau$ acts by
 \[\tau(X,Y)=(\tau_1(Y),\tau_2(X)),\quad \tau_1,\tau_2\in\text{Out}(\mathfrak h).
 \] 
 Similarly $\mathfrak g^\vee$ will have two factors isomorphic to $\mathfrak h^\vee$ and $\tau^\vee$ acts by $\tau^\vee(X^\vee,\mathcal Y^\vee)=(\tau_1^\vee(X^\vee),\tau_2^\vee(Y^\vee)).$ Then
 \[
 d(\tau^\vee(\mathcal O_1^\vee,\mathcal O_2^\vee))=(d\tau_1^\vee(\mathcal O_2^\vee),d\tau_2^\vee(\mathcal O_1^\vee))=(\tau_1 d(\mathcal O_2), \tau_2 d(\mathcal O_1))=\tau d (\mathcal O_1,\mathcal O_2).
 \]
This completes the proof.
    
\end{proof}

\subsection
{\bf Twisted Levi subgroup $\G'$ and $^L\!{\G'}$}

\begin{numbering}\rm 
Let $\bG'$ be an $E$-twisted Levi subgroup of $\bG$ where $E$ is a finite extension of $F$, that is,  $\bG'\otimes_F E$ is a $E$-Levi subgroup of $\bG\otimes_F E$. If $E$ is tamely ramified over $F$, we say that $\G'$ is a tame twisted Levi subgroup of $\G$.
\end{numbering}

\begin{lemma} \cite[Lemma 5.2.8]{Kal}
Let $\bG'\subset\bG$ be a tame twisted Levi subgroup. Let $\wGp\rightarrow\wG$ be the natural inclusion of dual groups, well defined up to $\wG$-conjugacy. There exists an extension of $\wGp\rightarrow\wG$ to an $L$-embedding $\LdualGp\rightarrow\LdualG$.
\end{lemma}

\begin{thm} Let $\G$ and $\G'$ be as above. Let 
\[
d':\left({\wGp}\backslash\nil_{\wGp}\right)
\longrightarrow 
\left(\bG'\backslash\nil_{\bG'}\right)
\]
be the duality map defined in Proposition \ref{prop: nil duality}.
\begin{enumerate}
    
\item $d'$ induces
\[
d':\left({\wGp}\backslash\nil_{\wGp}\right)^{W_F}\longrightarrow 
\left(\bG'\backslash\nil_{\bG'}\right)^{W_F}
\]
where $\left({\wGp}\backslash\nil_{\wGp}\right)^{W_F}$ (resp. $\left(\bG'\backslash\nil_{\bG'}\right)^{W_F}$) is the Weil group $W_F$ invariant nilpotent orbits in the Lie algebra of $\wGp$ (resp. $\bG'$).

\item Suppose $char(F)=0$. If $\G'$ is quasi-split, each orbit in $\left(\bG'\backslash\nil_{\bG'}\right)^{W_F}$ contains an element an $F$-rational point. That is, we have 
\[
\orb\cap\lieG'\neq\emptyset \quad\textrm{for}\ \orb\in\left(\bG'\backslash\nil_{\bG'}\right)^{W_F}
\]
\end{enumerate}
\end{thm}
\proof
(1) By Proposition \ref{prop: nil duality}, $d'$ is equivariant under the Galois action. Hence, $d'$ maps $W_F$-invariants to $W_F$-invariants.

(2) follows from the proof of the following theorem of Kottwitz's.

\begin{thm}\cite[Thm 4.2]{Kot} \ 
If $char(F)=0$ and $\bH$ is quasi-split, then every unipotent conjugacy class of $\bH$ that is defined over $F$ contains an element of $\bH(F)$.
\end{thm}



\begin{numbering}\label{num: example GL}{\bf Example. }\rm 
Let $G'=\bGL_m(E)$ be a twisted Levi subgroup of $G=\bGL_n(F)$ where $n=m\cdot[E:F]$. Then,
\[
\bG'(\overline F)=\prod_{i=1}^{[E:F]} \bGL_m(\overline F),
\quad \LdualGp=\left(\prod_{i=1}^{[E:F]} \bGL_m(\bbC)\right)\rtimes W_F
\]
where $W_F$ is the Weil group of $F$. Recall that the nilpotent orbits on $GL_m$ is parameterized by partitions of $m$.
Then the duality map 
\[
d':\bGL_m(\bbC)\backslash\nil_{\bGL_m(\bbC)}
\longrightarrow \bGL_m(E)\backslash\nil_{\bGL_m(E)}
\]
corresponds to the involution $\lambda\mapsto {}^t\lambda$ on partitions of $m$. 
\end{numbering}

\section
{\bf Types and $\sgss$-expansions}\label{sec: Ktype}

In this section, we assume $char(F)=0$.

\subsection{\bf Refined minimal $\sfK$-types}\ 

In \cite{HM} (cf. \cite{BK, BD}), a set of refined minimal $\sfK$-types $(J,\varrho)$ was constructed such that for any $\frs\in\frB(\G_n)$, there corresponds a $(J,\varrho)$ with the property that the
subcategory $\CaR^\frs(\G_n)$ is equivalent to the category $\CaR_\varrho(\G_n)$ of smooth representations of $\G_n$ which are generated by $\varrho$-isotypic components. Here, $J$ is a certain  compact open subgroup of $G_n$ and $\varrho$ is an irreducible smooth $J$-representation. Then one associates the Hecke algebra $\CaH(\varrho)$ to $(J,\varrho)$ as follows:
\[
\CaH(\varrho):=\{f\in\Ccs(\G_n,\End(\tilde\varrho))\mid f(jgj')=\tilde\varrho(j)f(g)\tilde\varrho(j')\}
\]
where $\tilde\varrho$ is the contragredient of $\varrho$.
Then, we have that the category of nondegnerate modules of $\CaH(\varrho)$ is equivalent to $\CaR^\frs(\G_n)$. Moreover, $\CaH(\varrho)$ is isomorphic to the Iwahori Hecke algebra of a twisted Levi subgroup as follows:

\begin{thm}\label{thm: Howe-Moy} (\cite[Theorem 5.6]{HM})
\begin{enumerate}
\item
Every $\pi\in\irrG$ contains a refined minimal $\sfK$-type.
\item 
Let $(J,\varrho)$ be a refined minimal $\sfK$-type. Then there exist extension $E_i/F$ and $n\in\bbN$, $1\le i\le u$, such that $n=\sum_{i=1}^u n_i[E_i:F]$, $J'=J\cap G'$ is an Iwahori subgroup of $\G'=\prod_{i=1}^u\GL_{n_i}(E_i)$, and there exists a Hecke algebra isomorphism
\[
\iota:\CaH':=\CaH(\G'//J')\longrightarrow
\CaH(\varrho):=\CaH(\G//J,\tilde\varrho)
\]
such that
\[
\supp(\iota(f))=J\supp(f)J,\ 
\supp(\iota(f))\cap\G'=\supp(f),\  f\in\CaH'.
\]
Futhermore, $\iota$ is an $L^2$-isometry for the natural $L^2$-structures on $\CaH(\varrho)$ and $\CaH'$.
\end{enumerate}
\end{thm}

\begin{defn}\label{def: pure type} \rm\cite[p456]{Mur}
A refined minimal $\sfK$-type $(J,\varrho)$ is {\it pure} if $u=1$.
\end{defn}

\medskip

\begin{remarks}\rm\ 
\begin{enumerate}
\item 
$(J,\varrho)$ is pure if the corresponding $\frs\in\frB$ satisfies that $\supp(\frs) \subset\{\nu^\alpha\rho\mid \alpha\in\bbC\}$ for some supercuspidal representation $\rho\in\CaC_m$. In this case, $\G'=\GL_{\frac nm}(E)$ for a degree $m$-extension $E$ of $F$.
\item 
Since the local Langlands correspondence and the decomposition of Hecke algebras (Theorem \ref{thm: Howe-Moy}) are compatible with parabolic inductions (see \S\ref{subsec: LLC}), it is enough to consider representations containing a pure refined minimal $\sfK$-types.
\end{enumerate}
\end{remarks}

\subsection{\bf $\sgss$-asymptotic expansion}

\begin{defn}\rm
Let $\gss\in\lieG_\sms$ be a semisimple element. Let $\G'=C_\G(\gss)$. Let $\orb(\gss)$ (resp. $\orb'(\gss)$) be the set of $\G$-orbits (resp. $\G'$-orbits) in $\lieG$ (resp. $\lieG'$)which contain $\gss$ in their closure. Note that there is one-to-one correspondence between $\orb'(\gss)$ and the set of nilpotent orbits $\orb'(0)$ in $\lieG'$.
\end{defn}

\begin{defn}\label{def: endscopic wf}\rm\ 
\begin{enumerate}
\item
Let $\gss\in\lieG_\sms$. Let $\eta$ be an exponentiable neighborhood of $0$. For $(\pi,V_\pi)\in\irrG$, we say that $\pi$ {\it is $\gss$-asymptotic on $\eta$} if the following holds for any $f\in\Ccs(\eta)$:
\begin{equation}\label{eq: gss-asymptotic}
\Theta_\pi(f\circ\exp)=\sum_{\orb\in\orb(\gss)}c_\orb\widehat\mu_\orb(f)
\end{equation}
for some $c_\orb\in\bbC$, $\orb\in\orb(\gss)$.
\item Recall $\orb(\gss)$ is partially ordered with respect to the closure inclusion relation. Define \emph{$\gss$-wavefront set} $\WF(\pi,\gss)$ as the set of maximal elements in $\{\orb\in\orb(\gss)\mid c_\orb\neq0\ \textrm{in }(\ref{eq: gss-asymptotic})\}$.
\end{enumerate}
\end{defn}

\begin{remarks}\rm\ Suppose $p$ is sufficiently large as in \cite{De} or \cite{KM}.
\begin{enumerate}
\item Any $(\pi,V_\pi)\in\irrG$ is $0$-asymptotic on $\lieG_{\dth(\pi)^+}$ (\cite{De}).
\item Under some assumptions on $F$, for any $(\pi,V_\pi)\in\irrG$, there is $\gss_\pi\in\lieG_\sms$ such that $\pi$ is $\gss_\pi$-asymptotic on $\lieG_{\frac{\dth(\pi)}2^+}$. Here for $r\in\bbR$, $\lieG_{r^+}=\cup_{s>r}\lieG_s$ where $\lieG_s$ is the union of $\lieG_{x,s}$ with $x$ running over the affine building of $\G$. It is an open and closed $\Ad(\G)$-invariant subset of $\lieG$. Note that if $\lieG_{\frac{\dth(\pi)}2^+}\supset\lieG_{\dth(\pi)+}$ (\cite{KM}). 
{In fact, $\gss_\pi$ depends only on the inertia class $\frs$ (thus $\varrho$). We will also write $\gss_\varrho$, $\gss_\frs$ for $\gss_\pi$ depending on the context.}
\item In (2), let $\G'=C_{\G}(\gss)$. Let $\nil'$ be the set of nilpotent elements in $\lieG'$. 
\end{enumerate}
\end{remarks}

{\begin{remark}\label{rmk: refined ss} \rm In the following, $s_\varrho$ is a semisimple element which is more refined than $\gss_\varrho$ in (2) above. More precisely, $\refs_\varrho$ is of the form $\gss_\varrho+\gss_0$ where $\gss_0\in (C_\lieG(\gss_\varrho))_{0^+}$ such that $C_G(s_\varrho)=\G'$. In \cite{Mur}, $s_\pi$ is used for $\refs_\varrho$.
\end{remark}}

\begin{thm}\label{thm: Murnaghan} (\cite[Thm 14.1]{Mur} and \cite{KM})
Let $(J,\varrho)$ be a pure refined minimal $\sfK$-type.
Suppose $\pi\in\irrG$ contains $\varrho$, that is, $\Hom_J(\varrho,\pi|_J)\neq 0$.
\begin{enumerate}
\item
If $\pi$ is supercuspidal, there exists $\gss_\varrho\in\lieG_{-\dth(\pi)}$ such that $Z_\G(\refs_\varrho)=\G'$ and 
\[
\Theta_\pi(\exp(X))=\dim(\varrho)\int_J\psi(\Tr(\refs_\varrho\Ad k(X)))\, dk
\]
for all $X\in\lieG_{\frac{\dth(\pi)}2^+}$. 
\item
In general, there exist $\refs_\varrho$ and coefficients $c_\orb(\pi)$, $\orb\in\orb(\refs_\varrho)$ such that
\[
\charpi(\exp(X))=\sum_{\orb\in\orb(\refs_\varrho)} c_\orb(\pi)\widehat\mu_\orb(X),\qquad
X\in\lieG_{\frac{\dth(\pi)}2^+}.
\]
\item
Let $\pi'$ be the irreducible unipotent representation of $\G'$ that corresponds to $\pi$ via Hecke algebra isomorphism $\iota$. Let 
\[
\Theta_{\pi'}=\sum_{\orb'\in\orb'(0)}c_{\orb'}\widehat\mu_{\orbp}.
\] 
For each $\orbp\in\orb'(0)$, let $\orb=\G\cdot(\refs_\varrho+\orbp)$. Then, $c_\orb(\pi)\neq0$ if and only if $c_{\orbp}(\pi')\neq0$.
\end{enumerate}
\end{thm}

{\begin{remark}\rm
In \cite{Mur}, a precise relation between coefficients are given
\[
c_\orb(\pi)=\frac{\vol_{G'}(J')}{\vol_G(J)}\cdot\dim(\varrho)\cdot c_{\orbp}(\pi')
\]
Combining with the result of \S\ref{sec: KL}, $c_\orb(\pi)$ is also computable.
\end{remark}
}

\begin{thm}\label{thm: Procter}(\cite{Pr}) Let $(J,\varrho)$ be a pure refined minimal $\sfK$-type.
Let $\pi'$ be the irreducible unipotent representation of $\G'$ that corresponds to $\pi$ via Hecke algebra isomorphism $\iota$ in \ref{thm: Howe-Moy}. Then, the Aubert-Zelevinsky involutions $\AZ(\pi')$ and $\AZ(\pi)$ of $\pi'$ and $\pi$ also correspond to each other via the isomorphism $\iota:\CaH'\longrightarrow\CaH(\varrho)$ in Theorem \ref{thm: Howe-Moy}. 
\end{thm}

In the following theorem, $\Irr_\varrho(G)$ (resp. $\Irr_{I'}(\G')$) denotes the set of irreducible representations of $G$ containing $\varrho$ (resp. irreducible unipotent representations of $G'$), and $\Rep_\varrho(W'_F)$ (resp. $\Rep_{I'}(W'_E)$) is the set of Langlands parameters  corresponding to $\Irr_\varrho(G)$ (resp. $\Irr_{I'}(\G')$) via the local Langlands correspondence. 

\begin{thm}\label{thm: commutative diagram}
The following diagram commutes: 
\[
\begin{array}{ccccccc}
\orb(\G',0)&\ \overset{\WF}\longleftarrow\ &\ \Irr_{I'}(\G')&\ \overset{\iota}\longleftrightarrow\ &\Irr_\varrho(G)\ 
&&\\
\phantom{d'}\uparrow&&\phantom{\AZ}\uparrow&&\uparrow\phantom{\AZ}&&\\
d'\,|&&\AZ\,\,|&&|\,\,\AZ&&\\
\phantom{d'}\,|&&\phantom{\AZ}\downarrow&&\downarrow \phantom{\AZ}&& \\
\Rep_{I'}(W_E')&\ \overset{LLC}\longleftrightarrow\ &\ \Irr_{I'}(\G')&\ \overset\iota\longleftrightarrow\ &\Irr_\varrho(G)\ & &
\end{array}
\]
Here, $d':\Rep_{I'}(W_E')\rightarrow\orb'(0)$ is defined as $d'(\sigma',N')=d'(N')$ where 
$d':\orb^{\prime\vee}\rightarrow\orb'$ is the duality map in (\ref{num: example GL}).
\end{thm}
\proof The first square commutes by \cite{CMBO} and \cite{MW}. The second square commutes by Theorem \ref{thm: Procter}. 
\qed

\begin{corollary} For any $\pi\in\Irr_\varrho(G)$, let $\pi'\in\irrGp$ be the corresponding unipotent representation of $G'$. Then,
\[
\WF(\pi,\gss)= d'(\CO'_{\AZ(\pi')})
\]
where $\AZ(\pi')$ is the Aubert-Zelevinsky dual of $\pi'$ and $\CO'_{\AZ(\pi')}$ is the $\G^{'\vee}$ is the $\G^{'\vee}$-saturation of the nilpotent part of Langlands parameter of $\AZ(\pi')$.
\end{corollary}
\proof
This follows from Theorems \ref{thm: Murnaghan} and \ref{thm: commutative diagram}.
\qed

\end{document}